\documentclass[11pt]{amsart}
\usepackage{amssymb}
\usepackage{amsthm}

\reversemarginpar

\newcommand{\C}[1]{\ensuremath{\mathcal{#1}}}
\newcommand{\ch}{\operatorname{char}}

\newcommand{\secref}[1]{Section~\ref{#1}}

\newtheorem{thm}{Theorem}[section] 

\theoremstyle{remark}

\numberwithin{equation}{section}  

\allowdisplaybreaks

\begin{document}

\title[Oriented Involutions on Group Rings]%
     {Oriented Involutions, Symmetric and Skew-Symmetric Elements in Group Rings}
\author{Edgar G. Goodaire \and C\'esar Polcino Milies}
\address{Memorial University\\
St. John's, Newfoundland\\
Canada A1C 5S7}
\email{edgar@mun.ca}
\thanks{The first author wishes to thank FAPESP
of Brasil for its support and the Instituto de Matem\'atica e Estat\'{\i}stica of
the Universidade de S\~ao Paulo for its usual wonderful hospitality
during the time that this research was conducted.}
\address{Instituto de Matem\'atica e Estat\'{\i}stica \\
Universidade de S\~ao Paulo, Caixa Postal 66.281 \\
CEP 05314-970, S\~ao Paulo SP \\
Brasil}
\email{polcino@ime.usp.br}
\thanks{This research was supported by a Discovery Grant
from the Natural Sciences and Engineering Research
Council of Canada, by FAPESP,
Proc. 09/52665-0 and 11/50046-1, and by CNPq., Proc. 300243/79-0(RN)
of Brasil.
\\ \today}

\begin{abstract}
Let $G$ be a group with involution $*$ and $\sigma\colon G\to\{\pm1\}$
a group homomorphism.  The map $\sharp$ that sends $\alpha=\sum\alpha_gg$
in a group ring $RG$ to $\alpha^{\sharp}=\sum\sigma(g)\alpha_gg^*$ is
an involution of $RG$ called an \emph{oriented group involution}.
An element $\alpha\in RG$ is \emph{symmetric} if $\alpha^{\sharp}=\alpha$
and \emph{skew-symmetric} if $\alpha^{\sharp}=-\alpha$.  The sets of symmetric
and skew-symmetric elements
have received a lot of attention in the special cases that $*$ is the inverse
map on $G$ and/or $\sigma$ is identically $1$, but not in general.  In this
paper, we determine the conditions under which the sets of elements
that are symmetric and skew-symmetric, respectively, relative to a general oriented
involution form subrings of $RG$.
The work on symmetric elements is a modification
and correction of previous work.
\end{abstract}

\maketitle

\section{Introduction}
Let $g\mapsto g^*$ denote an involution on a group $G$, let $\sigma\colon G\to\{\pm1\}$
be a group homomorphism and let $R$ be a commutative ring with $1$.
For an element $\alpha=\sum_{g\in G}\alpha_gg$ in the group
ring $RG$, define
\begin{equation*}
\alpha^{\sharp}=\sum_{g\in G}\sigma(g)\alpha_g g^*.
\end{equation*}
The map $\alpha\mapsto\alpha^{\sharp}$ is an involution of $RG$ called
an \emph{oriented group involution}, the function $\sigma$ being the \emph{orientation}.  In the
case that the involution on $G$ is the classical involution, $g\mapsto g^{-1}$, the map
$\sharp$ is precisely the oriented involution introduced by S.~P. Novikov in the context of
$K$-theory  \cite{Novikov:70}.

Oriented group involutions have been studied for a long time, perhaps starting with
the fundamental paper of Giambruno and Sehgal \cite{Giambruno:93}, but primarily in the special
cases that $g^*=g^{-1}$ or $\sigma(g)=1$ for all $g\in G$.  Much of the work in this area
is described in the comprehensive treatise by Gregory Lee \cite{Lee:10}.

Throughout this paper, we let
\begin{align*}
(RG)^+ &= \{\alpha\in RG\mid \alpha^{\sharp}=\alpha\} && \text{symmetric elements}\\
\intertext{and}
(RG)^- &= \{\alpha\in RG\mid \alpha^{\sharp}=-\alpha\} && \text{skew-symmetric elements.}
\end{align*}
As indicated, the elements of $(RG)^+$ are called \emph{symmetric}
and the elements of $(RG)^-$ are \emph{skew-symmetric}.
Historically, the questions of interest have revolved around the commutativity, anticommutativity
and Lie properties of the symmetric and skew-symmetric elements.
Some properties of \emph{unitary units} ($\alpha^\sharp=\pm\alpha^{-1}$)
have also been considered.  One might ask when the
sets $(RG)^+$ or $(RG)^-$ are subrings of $RG$.  It is not hard to show that these questions
are equivalent, respectively, to asking when the elements that are symmetric relative
to $\sharp$ commute and when the elements that are skew-symmetric relative to $\sharp$ anticommute.
Observing that $(RG)^+$ is a Jordan algebra under the operation $\alpha\circ\beta=
\alpha\beta+\beta\alpha$ and that $(RG)^-$ is a Lie algebra under the Lie bracket
$[\alpha,\beta]=\alpha\beta-\beta\alpha]$, it is natural to ask when these operations
are trivial.  These questions are equivalent, respectively, to asking when the elements
of $(RG)^+$ anticommute and when the elements of $(RG)^-$ commute.  All these questions
have been studied in special cases; the first question (when is $(RG)^+$ a subring)
was studied in general in \cite{Cristo:06d}, although this paper contained a slight error
which we correct here.
We also answer completely the question as to when $(RG)^-$ is a subring, thereby
settling the first two
of the four questions posed above.  \footnote{The last two questions are the subject
of partially completed research and will be the
focus of future papers.}

In this work, we make no assumptions
about $*$ and assume that the orientation $\sigma$ is not identically $1$.  This is equivalent
to the assumption that $\ker\sigma$ has index $2$ in the group $G$ and it implies
that the characteristic of the ring
is not $2$.  We assume throughout that $\sigma$ and $*$ are \emph{compatible}
in the sense that $\sigma(g^*)=\sigma(g)$ for all $g\in G$ and remark that this is the same as
requiring that both $\ker\sigma$ and its complement in $G$ be invariant under $*$.

It is not hard to find examples of oriented involutions where the orientation
and involution are not compatible.  For example, in
\begin{equation*}
D_n=\langle a,b\mid a^n=b^2=1,ba=a^{-1} b\rangle,
\end{equation*}
the dihedral group of order $2n$, the involution defined by $a^*=a^{-1}$, $b^*=ab$
and the orientation $\sigma$ defined by $\sigma(a)=-1$, $\sigma(b)=1$ are not compatible
(in characteristic different from $2$)
because $\sigma(ab)=-1$, whereas $\sigma\bigl((ab)^*\bigr)=\sigma(b^*a^*)=\sigma(aba^{-1})
=\sigma(a^2b)=1$.

On the other hand, compatibility does
not appear to be a strong assumption.  It is satisfied with any orientation when $*$
is the classical involution $g\mapsto g^{-1}$ and also
when $*$ is transpose and $\sigma$ is determinant on a group of matrices of
determinant $\pm1$.  It is satisfied  with $A^*=A^T$ and $\sigma(A)=\frac{\det A}{|\det A|}$
on any group of matrices.
Furthermore, computer explorations
with various groups suggest an abundance of oriented involutions where the orientation
is compatible with the involution, sixteen alone in the dihedral group $D_4$ of order~$8$
and twenty-four in the quaternion group of order~$8$.

A group $G$ is said to be \emph{SLC} if it has a unique nonidentity commutator, which we shall
always denote $s$,
and the \emph{limited commutativity} property, which says that the only way two
elements $g,h\in G$ can commute is with one of the three elements $g$, $h$, $gh$ central.
On such a group, the map $*$ defined by
\begin{equation}\label{eq0}
g^* = \begin{cases}
         g & \text{if $g$ is central} \\
         sg & \text{otherwise} \end{cases}
\end{equation}
is an involution (we refer to this as the \emph{canonical} involution on an SLC group)
with which any orientation is surely compatible.  We will often use (and sometimes implicitly)
the easily checked facts that such $s$ must be central, $s^*=s$ and $s^2=1$.
We refer the reader to \cite{EGG:96},  and especially to \S III.3.

\section{Skew-Symmetric Elements Anticommute}\label{sec3}

Let $(RG)^-=\{\alpha\in RG\mid \alpha^{\sharp}=-\alpha\}$ denote the set of elements of $RG$
that are skew-symmetric relative to an oriented group involution $\sharp$.  In this section,
we determine when this set forms a subring, equivalently, when the elements of $(RG)^-$ anticommute.

This question has been answered in the special case that the orientation is trivial
with this result \cite{EGG:09c} .  In characteristic different from $2$, either
\begin{itemize}
\item  $G$ is abelian and $*$ is the identity, or
\item $\ch R=4$, $G$ is abelian and there exists $s\in G$ with $s^2=1$
and $g^*=g$ or $sg$ for all $g\in G$, or
\item $\ch R=4$, $G$ is nonabelian with a unique nonidentity commutator, $s$ (necessarily
central of order $2$), and $g^*=g$ or $sg$ for all $g\in G$.
\end{itemize}
\medskip

Our theorem is this.

\begin{thm} Let $g\mapsto g^*$ denote an involution on a group $G$
and let $\sigma\colon G\mapsto\{\pm1\}$ be an orientation homomorphism which is not identically $1$.
Let $R$ be a commutative ring with $1$ and of characteristic different from $2$.
For $\alpha=\sum_{g\in G}\alpha_gg$
in the group ring $RG$, define $\alpha^{\sharp}=\sum_{g\in G}\alpha_g\sigma(g)g^*$.  Assume
$*$ and $\sigma$ are compatible in the sense that $\sigma(g)=\sigma(g^*)$ for all $g\in G$.
If the set $(RG)^-$ of
elements of $RG$ which are skew symmetric relative to $\sharp$ anticommute, then
\begin{enumerate}
\item $\ch R=4$, $G$ is abelian, $*$ is the identity on $N=\ker\sigma$ and $x^*\ne x$
for all $x\notin N$,
\par or
\item $\ch R=4$,  $G$ is an SLC group with  $*$ the canonical involution and $x^*\ne x$
for all $x\notin N$.
\end{enumerate}
Conversely, if $G$ is a group with an index $2$ subgroup $N$ and $\sigma\colon G\to\{\pm1\}$
is the homomorphism with kernel $N$,
then $(RG)^-$ is an anticommutative set given either of these situations.
\end{thm}
\begin{proof}
We begin by determining the nature of the elements of $(RG)^-$.
For this, it is convenient
to let $N=\ker\sigma$ and to partition the elements of $G$ into four sets,
\begin{align*}
S_1 &= \{n\in N\mid n^*=n\}, \\
S_2 &= \{n\in N\mid n^*\ne n\}, \\
S_3 &= \{x\notin N \mid x^*=x\}, \\
S_4 &= \{x\notin N \mid x^*\ne x\}.
\end{align*}
Any $\alpha\in RG$ can be written in the form
\begin{equation}\label{eq5}
\alpha = \sum_{g\in S_1} \alpha_gg + \sum_{g\in S_2} \alpha_gg
                    + \sum_{g\in S_3} \alpha_gg + \sum_{g\in S_4} \alpha_gg
\end{equation}
and then
\begin{equation*}
\alpha^{\sharp} = \sum_{g\in S_1} \alpha_gg + \sum_{g\in S_2} \alpha_gg^*
                    - \sum_{g\in S_3} \alpha_gg - \sum_{g\in S_4} \alpha_gg^*.
\end{equation*}
Thus
\begin{equation*}
\alpha^{\sharp}=-\alpha \text{ if and only if }
       \begin{cases}
            \alpha_g=-\alpha_g & \text{if $g\in S_1$} \\
            \alpha_g=-\alpha_{g^*} & \text{if $g\in S_2$} \\
            \alpha_g=\alpha_g & \text{if $g\in S_4$.}
            \end{cases}
\end{equation*}
Thus $(RG)^-$ is spanned over $R$ by elements of the sets
\begin{align*}
\C{S}_1 &= \{\alpha_n n\mid n^*=n\in N,  2\alpha_n=0\}, \\
\C{S}_2 &= \{n-n^*\mid n\in N, n^*\ne n\}, \\
\C{S}_3 &= \{x\notin N\mid x^*=x\}, \\
\C{S}_4 &= \{x+x^*\mid x\notin N, x^*\ne x\}
\end{align*}
and, of course, $(RG)^-$ is anticommutative if and only if any two elements from the
union of the $\C{S}_i$ anticommute.  If this is the case, and if $x\in\C{S}_3$, then $xx=-xx$
implies $2x^2=0$, which cannot happen in characteristic different from $2$.
Thus $\C{S}_3=\emptyset$ and  $x^*\ne x$ for every $x\notin N$.  In particular, $*$
is not the identity.

We begin with the converse, showing that in each of the situations
described in the theorem, the elements of $(RG)^-$ anticommute.

\par\medskip

(1) Suppose $G$ is abelian, $*$ is an involution on $G$ which moves every element outside $N$
but whose restriction to $N$ is the identity and $\ch R=4$.  Then $\C{S}_2=\{0\}$ and $\C{S}_3=\emptyset$.
Elements $\alpha_nn$ and $\alpha_mm$ of $\C{S}_1$ anticommute because $(\alpha_mm)(\alpha_nn)
=\alpha_m\alpha_nmn=-\alpha_m\alpha_nmn$ (using $2\alpha_m=0$)
$=-\alpha_n\alpha_mnm=-(\alpha_nn)(\alpha_mm)$.
For a similar reason, any element of $\C{S}_1$ anticommutes with
any element of $\C{S}_4$.  To see that two elements of $\C{S}_4$ anticommute, we compute
\begin{equation}\label{eq4}
(x+x^*)(y+y^*)+(y+y^*)(x+x^*)=2xy+2xy^*+2x^*y+2x^*y^*
\end{equation}
with $x\notin N$, $y\notin N$, $x^*\ne x$ and $y^*\ne y$.
Since $xy\in N$, we have $xy=(xy)^*=y^*x^*=x^*y^*$ and, similarly, $xy^*=x^*y$, so
the right side of \eqref{eq4} is $4xy+4x^*y=0$.
\par\medskip

(2) Suppose $\ch R=4$ and $G$ is an SLC group with canonical involution $*$ that
moves every element outside $N$.  Thus $x^*=sx$ if $x\notin N$.   Again $\C{S}_3=\emptyset$,
so it suffices to show
that any two elements of $\C{S}_1\cup \C{S}_2\cup\C{S}_4$ anticommute.  First we
note that any element $\alpha_nn$ of $\C{S}_1$ anticommutes with any element $\beta$ in
the group ring because $n^*=n$ means $n$ is central, so $(\alpha_nn)\beta=\beta(\alpha_nn)=
-\beta(\alpha_nn)$ because $2\alpha_n=0$.  Take $n-n^*$ and $m-m^*$ in $\C{S}_2$.  Then
$n-n^*=n-sn=(1-s)n$ and $m-m^*=(1-s)m$.  Since $(1-s)^2=1-2s+s^2=2(1-s)$,
\begin{align*}
(n-n^*)(m-m^*) &= 2(1-s)nm \\
&= \begin{cases}
                 2(1-s)mn & \text{if $nm=mn$} \\
                 2(1-s)smn =-2(1-s)mn & \text{if $nm\ne mn$}
                 \end{cases}
\end{align*}
and, because the characteristic is $4$, in either case we get $-2(1-s)mn=-(m-m^*)(n-n^*)$.
The situation is similar if we consider two elements $x+x^*=(1+s)x$ and $y+y^*=(1+s)y$
of $\C{S}_4$ because
\begin{align*}
(x+x^*)(y+y^*)  &= 2(1+s)xy \\
&= \begin{cases}
         2(1+s)yx & \text{if $xy=yx$} \\
         2(1+s)syx=2(1+s)yx & \text{if $xy\ne yx$}
   \end{cases}
\end{align*}
and again, because $\ch R=4$, $2(1+s)yx=-2(1+s)yx=-(y+y^*)(x+x^*)$.  Finally, we observe
that any element of $\C{S}_2$ anticommutes with any element of $\C{S}_4$ because the product
of two such elements is $0$:  $(n-n^*)(x+x^*)=(1-s)(1+s)nx=0$.

\medskip

Now we attack the first statement of the theorem assuming
that the elements of $(RG)^-$ anticommute and hence $\C{S}_3=\emptyset$.

The compatibility condition implies that $N$ is invariant under $*$, so we may apply to $N$
the results of \cite{EGG:09c} mentioned at the start of this section
in the case that the orientation is identically $1$. There are three possibilities.
\bigskip

1) Suppose first that $N$ is abelian and $*$ is the identity on $N$.
Let $x\notin N$.  Since $x^*\ne x$, $x+x^*\in\C{S}_4$.  The equation $(x+x^*)(x+x^*)=-(x+x^*)(x+x^*)$
implies $2(x+x^*)^2=0$ which, using $(x^*)^2=(x^2)^*=x^2$ because $x^2\in N$, gives
\begin{equation*}
2(2x^2+xx^*+x^*x)=0.
\end{equation*}
Now $x^2\ne xx^*$ and $x^2\ne x^*x$, so $x^*x=xx^*$ and $\ch R=4$.

Let $x$ and $y$ be any two elements of $G\setminus N$.  The elements
$x+x^*$ and $y+y^*$ are in $\C{S}_4$, so they anticommute.  The equation
$(x+x^*)(y+y^*)=-(y+y^*)(x+x^*)$ gives
\begin{equation}\label{eq3}
xy+xy^*+x^*y+x^*y^*+yx+yx^*+y^*x+y^*x^*=0.
\end{equation}
Now $xy\in N$, so $y^*x^*=(xy)^*=xy$.  Similarly, $xy^*=yx^*$, $x^*y=y^*x$, $x^*y^*=yx$
and \eqref{eq3} becomes $2(xy+xy^*+x^*y+yx)=0$. Since $xy\ne xy^*$ and $xy\ne x^*y$,
the only possibility is $xy=yx$; thus the elements of $G\setminus N$ commute.

Let $n\in N$
and $x\notin N$.  Then $nx\notin N$ so $(nx)x=x(nx)$, implying $nx=xn$.  Since $N$
is abelian, it follows that $G$ is abelian and we have the first situation described in the theorem.

\medskip
2) Assume $\ch R=4$, $N$ is abelian and there exists $s\in N$ with $s^2=1$
and $n^*=n$ or $sn$ for all $n\in N$.
We may assume that $n^*\ne n$ for some $n\in N$,
for otherwise we are in Case~1 which has  just been settled.
In particular, we may assume that $s\ne 1$ and then, since $s^*=s$ or $s^*=ss=1$, which is false,
we get $s^*=s$.

Let $x\notin N$.  Then $x^*\ne x$ and $x+x^*\in\C{S}_4$ anticommutes with itself giving
$x^2=(x^*)^2=(x^2)^*$, $xx^*=x^*x$, as before.
Furthermore,  $sx\notin N$, so $x+x^*$ and $(sx)+(sx)^*=sx+x^*s$ anticommute, that is,
\begin{align*}
0 &= (sx+x^*s)(x+x^*)+(x+x^*)(sx+x^*s) \\
&=sx^2+sxx^*+x^*sx+x^*sx^*+xsx+xx^*s+x^*sx+(x^*)^2s \\
&= 2sx^2+2sxx^*+2x^*sx+x^*sx^*+xsx,
\end{align*}
remembering that $N$ is abelian and noting that $xx^*$ and $x^2$ are in $N$.
Now $sx^2\ne sxx^*$ because $x\ne x^*$ and $sx^2\ne x^*sx$ because $\C{S}_3=\emptyset$ means
$sx\ne (sx)^*=x^*s$.  In characteristic $4$, it follows that $sx^2=x^*sx^*=xsx$, so $sx=xs$.
Thus $s$ is central in $G$.

Let $n$ be any element of $N$ with $n^*=n$.  So $2n\in\C{S}_1$.  Let $x$
be an element of $G$ not in $N$.  Then $x+x^*\in\C{S}_4$, so $2n$
and $x+x^*$ anticommute, giving
\begin{equation*}
2nx+2nx^*+2xn+2x^*n=0.
\end{equation*}
Now $nx\ne nx^*$ and $nx\ne (nx)^*=x^*n$, so $nx=xn$.  This shows that $n$ commutes with all elements
not in $N$.  Since $n$ also commutes with elements in $N$ ($N$ is abelian),
this element is central in $G$.
In particular, for any $x\notin N$, both $xx^*$ and $x^2$ are central.

Suppose $n\in N$ and $n^*=sn$.  Let $x$ be an element not in $N$.
Then $nx$ is not in $N$ (so $nx\ne (nx)^*=sx^*n$)
and $nx+(nx)^*=nx+sx^*n$ anticommutes with $x+x^*$ giving
\begin{equation*}
nx^2+nxx^*+sx^*nx+sx^*nx^*+xnx+sxx^*n+x^*nx+s(x^*)^2n=0.
\end{equation*}
In characteristic $4$, two sets of four elements on the left must consist of equal elements.
Now $nx^2\ne nxx^*$
and  $nx^2\ne sx^*nx$ (because $nx\ne (nx)^*$), so three elements of $\{sx^*nx^*,xnx,sxx^*n,x^*nx,s(x^*)^2n
=sx^2n\}$ are equal and equal to $nx^2$.  Since $xnx\ne x^*nx$ and $sxx^*n\ne sx^2n$, there
are four possibilities:
\begin{enumerate}
\item[i.] $nx^2=sx^*nx^*=xnx=sxx^*n$,
\item[ii.] $nx^2=sx^*nx^*=xnx=sx^2n$,
\item[iii.] $nx^2=sx^*nx^*=x^*nx=sxx^*n$,
\item[iv.] $nx^2=sx^*nx^*=x^*nx=sx^2n$.
\end{enumerate}
We have previously observed that $x$ and $x^*$ commute, so $xx^*n=x^*xn$.
Thus Case~i implies both $nx=xn$ and $nx^*=xn$, so $x^*=x$, which is wrong.  Case~ii says
$nx^2=sx^2n=snx^2$ (squares are central), an obvious contradiction, and Case~iv gives
the same contradiction, so we must be in Case~iii.
Here $x^2n=nx^2=sxx^*n$ gives $x=sx^*$, hence $x^*=sx$, and then $nx^2=x^*nx=sxnx$
gives $nx=sxn$.  As a biproduct of these arguments, we have learned also
that $n^*\ne n$ implies $n$ is not central, so
\begin{equation*}
\C{Z}(G)=\{n\in N\mid n^*=n\}
\end{equation*}
and, furthermore, the commutator of an element of $N$ with an element not in $N$ is $1$ or $s$.

Fix an $a\notin N$ and, with $n,m\in N$, use $G=N\cup Na$ to find the commutator of two elements $x=na$,
$y=ma$, neither of which is in $N$.  On the one hand
\begin{equation*}
xy=(na)(ma)=\begin{cases}
                nma^2 & \text{if $am=ma$} \\
                snma^2 & \text{if $am=sma$}
                \end{cases}
\end{equation*}
while, similarly, $yx=mna^2$ or $smna^2$.  Since $(m,n)=1$ or $s$, so also $(x,y)$ is $1$ or $s$.

Our reasonings have shown that $G'=\{1,s\}$ and that
\begin{equation*}
g^*=\begin{cases}
      g & \text{if $g\in\C{Z}(G)$} \\
      sg & \text{if $g\notin\C{Z}(G)$}.
            \end{cases}
\end{equation*}
Since $*$ is an involution, $G$ must have the LC property, for this reason:
if $gh=hg$ with none of $g,h,gh$ central, then $sgh=(gh)^*=h^*g^*=(sh)(sg)=hg=gh$,
which cannot be.  Thus we are in the second situation described by the theorem.

\medskip
3) Finally, assume that $\ch R=4$ and that $N$ is a nonabelian group with a unique nonidentity commutator,
$s$ (necessarily central in $N$ and of order $2$), and $n^*=n$ or $sn$ for every $n\in N$.
As an involution on a nonabelian group, $*$ cannot be the identity on $N$,
so $s\ne1$ and, as in the previous case, we have $s^*=s$.  Moreover, as in previous cases,
we have $x^2=(x^2)^*$ and $xx^*=x^*x$ for any $x\notin N$.

Let $x\notin N$.  Then also $sx\notin N$ so $x+x^*$ and $(sx)+(sx)^*=sx+x^*s$ anticommute, that is,
\begin{align*}
0 &= (sx+x^*s)(x+x^*)+(x+x^*)(sx+x^*s) \\
&=sx^2+sxx^*+x^*sx+x^*sx^*+xsx+xx^*s+x^*sx+(x^*)^2s \\
&= 2sx^2+2sxx^*+2x^*sx+x^*sx^*+xsx
\end{align*}
using the facts that $xx^*$ and $(x^*)^2=x^2$ are in $N$ and hence commute with $s$,
which is central in $N$.
Now $sx^2\ne sxx^*$ because $x\ne x^*$ and $sx^2\ne x^*sx$ because $\C{S}_3=\emptyset$ means
$sx\ne (sx)^*=x^*s$.  In characteristic $4$, it follows that $sx^2=x^*sx^*=xsx$, so $sx=xs$.
Thus, not only is $s$ central in $N$, but it is actually central in $G$.

Let $n$ be any element of $N$ satisfying $n^*=n$.  Then $2n\in\C{S}_1$.  Let $x$
be an element of $G$ not in $N$.  Then $x+x^*\in\C{S}_4$, so $2n$
and $x+x^*$ anticommute, that is,
\begin{equation*}
0=2nx+2nx^*+2xn+2x^*n.
\end{equation*}
Now $nx\ne nx^*$ and $nx\ne (nx)^*=x^*n$ (remember that $\C{S}_3=\emptyset$), so $nx=xn$.
This shows that
$n$ commutes with all elements of $G\setminus N$.
Take any $m$ in $N$ and $x\notin N$.  Then $n$ and $mx$ commute, so $nmx=mxn=mnx$ so $mn=nm$.
This proves that if $n^*=n\in N$, then $n$ is central in $G$.  In particular, for any $x\notin N$,
we have both $xx^*$ and $x^2$ central in $G$.

Since $*$ is not the identity on $N$, there exists $n\in N$
with $n^*=sn$.  Let $x$ be an element of $G$ not in $N$.  Then $nx$ is not in $N$ (so $nx\ne (nx)^*=sx^*n$)
and $nx+(nx)^*=nx+sx^*n$ anticommutes with $x+x^*$ giving
\begin{equation*}
0 = nx^2+nxx^*+sx^*nx+sx^*nx^*+xnx+sxx^*n+x^*nx+s(x^*)^2n.
\end{equation*}
Now $nx^2\ne nxx^*$
and  $nx^2\ne sx^*nx$ (because $nx\ne (nx)^*$), so three elements of $\{sx^*nx^*,xnx,sxx^*n,x^*nx,s(x^*)^2n
=sx^2n\}$ are equal and equal to $nx^2$.  Now $xnx\ne x^*nx$ and $sxx^*n\ne sx^2n$, so there
are apparently four possibilities:
\begin{enumerate}
\item[i.] $nx^2=sx^*nx^*=xnx=sxx^*n$,
\item[ii.] $nx^2=sx^*nx^*=xnx=sx^2n$,
\item[iii.] $nx^2=sx^*nx^*=x^*nx=sxx^*n$,
\item[iv.] $nx^2=sx^*nx^*=x^*nx=sx^2n$.
\end{enumerate}
Case i implies both $nx=xn$ and $nx^*=xn$ (using $sxx^*n=sx^*xn$) and hence $x=x^*$, a contradiction.
Cases ii and iv say $nx^2=sx^2n=snx^2$ by centrality of $x^2$, an obvious contradiction, so the
situation is as described in Case~iii.  So $x^2n=nx^2=sxx^*n$ and hence $x=sx^*$ and $x^*=sx$,
and then $nx^2=x^*nx=sxnx$, so $nx=sxn$.
In particular, we learn that $n^*\ne n$ implies $n$ is not central, so, just as before,
\begin{equation*}
\C{Z}(G)=\{n\in N\mid n^*=n\}.
\end{equation*}
We have also seen that any commutator $(n,x)$, with $n\in N$ and $x\notin N$, is $1$ or $s$.
The commutator of any two elements of $N$ is also $1$ or $s$.
Fix an $a\notin N$, take $n,m\in N$ and use $G=N\cup Na$ to find the commutator of two elements $x=na$,
$y=ma$ neither of which is in $N$.  On the one hand
\begin{equation*}
xy=(na)(ma)=\begin{cases}
                nma^2 & \text{if $am=ma$} \\
                snma^2 & \text{if $am=sma$}
                \end{cases}
\end{equation*}
while, similarly, $yx=mna^2$ or $smna^2$.  Since the commutator $(n,m)=1$ or $s$, it is clear
that $(x,y)$ is also $1$ or $s$.  All this shows
that $G'=\{1,s\}$ and the involution $*$ is defined by
\begin{equation*}
g^*=\begin{cases}
      g & \text{if $g\in\C{Z}(G)$} \\
      sg & \text{if $g\notin\C{Z}(G)$}.
      \end{cases}
\end{equation*}
It remains just to show that $G$ has the limited commutativity (LC) property (so that $G$ is SLC).

Suppose two elements $x\notin N$, $y\notin N$ commute.  Then $xy=yx$ implies $xyx^*=yxx^*=xx^*y$,
$yx^* = x^*y$ and, similarly, $xy^*=y^*x$.   Now $x+x^*$ and $y+y^*$ are in $\C{S}_4$,
so they anticommute, giving
\begin{align*}
0 &= (x+x^*)(y+y^*)+(y+y^*)(x+x^*) \\
&= 2xy++2x^*y^*+2xy^*+2x^*y.
\end{align*}
The only possibility is $xy=x^*y^*$ and $xy^*=x^*y$.  The former equation says
$xy=(yx)^*=(xy)^*$ so $xy$ is central.

We have seen that $nx=xn$, $n\in N$, $x\notin N$, implies that $n$ is not central.
Finally, if $n,m\in N$ commute and neither is central, then $n^*=sn$, $m^*=sm$
and so $(nm)^*=m^*n^*=(sm)(sn)=mn=nm$, that is, $nm$ is central.  Thus $G$ indeed
has the LC property, the situation is as described in (2), and the proof is complete.
\end{proof}

\section{Symmetric Elements Commute}\label{sec2}

In this section, we find conditions under which the $\sharp$-symmetric elements of a
group ring $RG$ commute and thus determine when $(RG)^+$ is a subring
of $RG$.
As mentioned, this problem has been previously studied \cite{Cristo:06d} although, unfortunately,
the possible existence of $2$-torsion elements in $R$ was overlooked.
The correct result is this.

\begin{thm} Let $g\mapsto g^*$ denote an involution on a group $G$
and let $\sigma\colon G\mapsto\{\pm1\}$ be an orientation homomorphism which is not identically $1$.
Set $N=\ker\sigma$.
Let $R$ be a commutative ring with $1$ and let $R_2=\{r\in R\mid 2r=0\}$.
For $\alpha=\sum_{g\in G}\alpha_gg$
in the group ring $RG$, define $\alpha^{\sharp}=\sum_{g\in G}\sigma(g)\alpha_gg^*$.  Assume
$*$ and $\sigma$ are compatible in the sense that $\sigma(g)=\sigma(g^*)$ for all $g\in G$.
If the set $(RG)^+$ of
$\sharp$-symmetric elements of $RG$ is commutative, then $G$ is abelian or else
one of the following situations arises:
\begin{enumerate}
\item[(1)] $R_2^2=\{0\}$, $N$ is abelian, $x^*=x$ for $x\in G\setminus N$ and
$n^*=a^{-1}na$ for any $n\in N$ and any $a\in G\setminus N$;
\item[(2)] the characteristic of $R$ is $4$, $N$ is abelian,
$G$ has a unique nonidentity commutator, $s$, $x^*=sx$ for $x\in G\setminus N$, and
$n^*=a^{-1}na$ for any $n\in N$ and any $a\in G\setminus N$;
\item[(3)] $R_2^2=\{0\}$,  $N$ is SLC with unique nonidentity
commutator $s$ and canonical involution $*$, and $G=N\langle a\rangle$ is the product of $N$ and a central
subgroup generated by an element $a$ satisfying $a^*=sa$.
\item[(4)] the characteristic of $R$ is $4$ and both $N$ and $G$ are SLC groups
with canonical involution.
\end{enumerate}
Conversely, suppose that $G$ is a group with an index $2$ subgroup $N$ and $\sigma\colon G\to\{\pm1\}$ is
the orientation homomorphism with kernel $N$.  Furthermore, suppose that either $G$ is abelian
with $*$ the identity map or
the conditions of (1), (2), (3) or (4) are satisfied.  Then $*$ is an involution on $G$ and
the set $(RG)^+$ of $\sharp$-symmetric elements is a commutative set.
\end{thm}
\begin{proof}
As before, we begin by determining the nature of the elements of $(RG)^+$.  With the
$S_i$ as in \secref{sec3} and $\alpha\in RG$ written in the form \eqref{eq5}, we have
\begin{equation*}
\alpha^{\sharp}=\alpha \text{ if and only if }
       \begin{cases}
            \alpha_g=\alpha_{g*} & \text{if $g\in S_2$} \\
            \alpha_g=-\alpha_g & \text{if $g\in S_3$} \\
            \alpha_g=-\alpha_g^* & \text{if $g\in S_4$.}
            \end{cases}
\end{equation*}
Thus $(RG)^+$ is spanned over $R$ by the union of the sets
\begin{align*}
\C{T}_1 &= \{n\in N\mid n^*=n\}, \\
\C{T}_2 &= \{n+n^*\mid n\in N, n^*\ne n\}, \\
\C{T}_3 &= \{\alpha_x x\mid x^*=x\notin N, 2\alpha_x=0\}, \\
\C{T}_4 &= \{x-x^*\mid x\notin N, x^*\ne x\}.
\end{align*}
Clearly, the elements of $(RG)^+$ commute if and only if any two elements
from the union of the $\C{T}_i$ commute.

We begin our proof with the converse, noting initially
that if $G$ is abelian, then $RG$ is a commutative ring, so the elements of $(RG)^+$ commute.

\smallskip
(1) Suppose $R_2^2=\{0\}$, $N$ is abelian, $x^*=x$ for $x\notin N$ (so $\C{T}_4=\emptyset$)
and $n^*=a^{-1}na$ for
any $n\in N$ and any $a\notin N$.
We first note that $n^*$ is well-defined because if $a$ and $b$ are any two elements of $G\setminus N$,
then $b=ma$
for some $m\in N$ because $N$ has index $2$ in $G$, so $b^{-1}nb=a^{-1}m^{-1}nma=a^{-1}na$
because $N$ is abelian.  Now fix $a\notin N$.
To show that $*$ is an involution, there are three cases to consider.
\begin{itemize}
\item If $n_1,n_2\in N$, then $(n_1n_2)^*=a^{-1}n_1n_2a=(a^{-1}n_1a)(a^{-1}n_2a)
=(a^{-1}n_2a)(a^{-1}n_1a)$ (because $N$ is abelian), and this is $n_2^*n_1^*$.
\item If $n_1\in N$ and $x\notin N$, then $x=n_2a$ for some $n_2\in N$.  Since $n_1x\notin N$,
$(n_1x)^*=n_1x$, while $x^*n_1^*=xa^{-1}n_1a=n_2aa^{-1}n_1a=n_2n_1a=n_1n_2a=n_1x$ also.
\item If $x\notin N$ and $y\notin N$, then $x=an_1$, $y=an_2$ for some $n_1,n_2\in N$.
Since $xy\in N$, $(xy)^*=a^{-1}xya=a^{-1}an_1an_2a=n_1an_2a$, whereas $y^*x^*=yx=an_2an_1$.
Now $an_2\notin N$ and $a\notin N$, so their product $an_2a$ is in $N$ and so commutes
with $n_1$.  This gives $y^*x^*=n_1an_2a=(xy)^*$ in this case as well.
\end{itemize}
To show that the elements of $(RG)^+$ commute, it is sufficient to show
that any two elements of $\C{T}_1\cup\C{T}_2\cup \C{T}_3$ commute ($\C{T}_4=\emptyset)$.
Any two elements of $\C{T}_1\cup\C{T}_2$ commute because $N$ is abelian.  Elements $\alpha_x x$
and $\alpha_y y$ of $\C{T}_3$ commute because, as elements of $R_2$, $\alpha_x$ and $\alpha_y$
have product~$0$.
Let $n\in \C{T}_1$ and $\alpha_x x\in\C{T}_3$.  Since the restriction of $*$ to $N$
is conjugation by any element
not in $N$, $n=n^*=x^{-1}nx$, so $xn=nx$ implying that $n$ and $\alpha_x x$ commute.  Finally,
if $n+n^*\in\C{T}_2$ and $\alpha_x x\in\C{T}_3$, then $n^*=x^{-1}nx$, so $x(n+n^*)=xn+nx$,
whereas $(n+n^*)x=nx+x^{-1}nx^2=nx+xn$ too
(because $x^2$ is in $N$ and so commutes with $n$).

\smallskip
(2) Suppose $\ch R=4$, $N$ is abelian, $G$ has a unique nonidentity commutator, $s$ (necessarily
central and of order $2$),  $x^*=sx$ for $x\notin N$ and $n^*=a^{-1}na$ for any $n\in N$ and any $a\notin N$.
As in our discussion of (1) just above, $n^*$ is well-defined and, to show that $*$ is an involution on $G$,
there are three cases to consider.
\begin{itemize}
\item If $n_1,n_2\in N$, then $(n_1n_2)^*=n_2^*n_1^*$ as before.
\item If $n_1\in N$ and $x\notin N$, then $x=n_2a$ for some $n_2\in N$.  Since $n_1x\notin N$,
$(n_1x)^*=sn_1x$, while $x^*n_1^*=sxa^{-1}n_1a=sn_2aa^{-1}n_1a=sn_2n_1a=sn_1n_2a=sn_1x$ also.
\item If $x\notin N$ and $y\notin N$, then $x=an_1$, $y=an_2$ for some $n_1,n_2\in N$.
Since $xy\in N$, $(xy)^*=a^{-1}xya=a^{-1}an_1an_2a=n_1an_2a$, whereas $y^*x^*=(sy)(sx)=yx=an_2an_1$.
As above, the elements $(xy)^*$ and $y^*x^*$ are the same.
\end{itemize}
This time, $\C{T}_3=\emptyset$, so to show that the elements of $(RG)^+$ commute, it suffices to show
that any two elements of $\C{T}_1\cup\C{T}_2\cup\C{T}_4$ commute.
Any two elements of $\C{T}_1\cup\C{T}_2$ commute because $N$ is abelian.
Let $x-x^*=x-sx=(1-s)x$ and $y-y^*=(1-s)y$ be two elements of $\C{T}_4$.  Then
$(x-x^*)(y-y^*)=(1-s)^2xy=2(1-s)xy$, whereas
\begin{equation*}
(y-y^*)(x-x^*)=2(1-s)yx=\begin{cases}
   2(1-s)xy & \text{if $yx=xy$} \\
   2(1-s)sxy &\text{if $yx\ne xy$}.
   \end{cases}
\end{equation*}
Since $2(1-s)s=2(s-1)=-2(1-s)=2(1-s)$ (in characteristic $4$), we see that $x-x^*$ and $y-y^*$
commute in either case.
Let $n\in \C{T}_1$ and $x-x^*\in \C{T}_4$.   Since $n=n^*=x^{-1}nx$,  $n$ and $x$ commute, so
$n$ and $x-x^*=(1-s)x$ commute.  Finally, let $n+n^*\in\C{T}_2$ and
$x-x^*=(1-s)x\in\C{T}_4$.  If $nx=xn$, then $(nx)^*=(xn)^*$, that is, $x^*n^*=n^*x^*$
which says $sxn^*=sn^*x$, so $n^*x=xn^*$ too.  Then $(n+n^*)(x-x^*)=(1-s)(nx+n^*x)$, whereas
$(x-x^*)(n+n^*)=(1-s)(xn+xn^*)$ and these elements are the same.  If $nx\ne xn$,
then $nx=sxn$, $n^*=x^{-1}nx=sn$, $n+n^*=(1+s)n$ and both products
$(n+n^*)(x-x^*)$ and $(x-x^*)(n+n^*)$ are $0$ because $(1+s)(1-s)=0$.  In any case,
an element of $\C{T}_2$ always commutes
with an element of $\C{T}_4$.
\smallskip

(3) Suppose $R_2^2=\{0\}$, $N$ is SLC with unique nonidentity
commutator $s$ and canonical involution $*$,
and $G=N\langle a\rangle$ is the product of $N$ and a cyclic subgroup
$\langle a\rangle$ generated by a central element $a\notin N$ satisfying $a^*=sa$.
Since $G=N\cup Na$, if $x\notin N$, there exists $n\in N$ with $x=na$, so $x^*=a^*n^*=sn^*a$ (because
$a$ is central).  Thus $*$ extends to a map $G\to G$ which is an involution on $G$
because $*$ is an involution on $N$,
\begin{itemize}
\item $[n_1(n_2a)]^*=[(n_1n_2)a]^*=s(n_1n_2)^*a=sn_2^*n_1^*a=sn_2^*an_1^*$ ($a$ is central)
$=(n_2a)^*n_1^*$,
\par and
\item $[(n_1a)(n_2a)]^*=[(n_1n_2)a^2]^*=(a^2)^*(n_1n_2)^*=(a^*)^2(n_1n_2)^*$
\par $=(sa)^2(n_1n_2)^*=(n_1n_2)^*a^2$ because $s^2=1$,
\par while $(n_2a)^*(n_1a)^*=sn_2^*asn_1^*a=(n_1n_2)^*a^2$ too.
\end{itemize}
To show that the elements of $(RG)^+$ commute, we first remark that elements $n\in N$ with $n^*=n$
and of the form $n+n^*$, $n\in N$, are known to be central in the group ring $RN$
because $N$ is SLC \cite[Corollary III.4.3]{EGG:96}.
It follows that any two elements of $\C{T}_1\cup\C{T}_2$ commute.
Elements $\alpha_xx$ and $\alpha_yy$
in $\C{T}_3$ commute because $\alpha_x\alpha_y=0$.  If $x-x^*$ is in $\C{T}_4$, then $x=na$
for some $n\in N$, so $x-x^*=na-sn^*a=(n-sn^*)a\in (RN)a$.  Now consider two elements
$(n-sn^*)a$ and $(m-sm^*)a$ of this form.  Commutativity is equivalent to
$(n-sn^*)(m-sm^*)=(m-sm^*)(n-sn^*)$,
that is,
\begin{equation*}
nm-snm^*-sn^*m+n^*m^*=mn-smn^*-sm^*n+m^*n^*.
\end{equation*}
If $n$ or $m$ is central, say $n$, then $n^*=n$ and this equation is easily seen  to be satisfied.
On the other hand, if neither $n$ nor $m$ is central, then $n^*=sn$, $m^*=sm$ and each side of the equation
is $0$.  In any case, two elements of $\C{T}_4$ commute.
Commutativity of all other pairs of
elements in $\C{T}_1\cup\C{T}_2\cup\C{T}_3\cup\C{T}_4$ follows from centrality of~$a$.
\smallskip

(4) Suppose $\ch R=4$ and both $N$ and $G$ are SLC groups with canonical involution.
To show that the elements of $(RG)^+$ commute, it suffices to show
that any two elements of $\C{T}_1\cup\C{T}_2\cup\C{T}_3\cup\C{T}_4$ commute.
In an SLC group $G$, $*$-symmetric elements
and elements of the form $g+g^*$ are central \cite[Corollary III.4.3]{EGG:96},
so the elements of  $\C{T}_1\cup\C{T}_2\cup\C{T}_3$ are central.  It remains then only
to consider two elements of $\C{T}_4$ and these commute with precisely the argument
given above in case (2).
\bigskip

We now attack the first statement of the theorem, assuming that $G$ is nonabelian and
$\alpha\mapsto\alpha^\sharp$ is a nontrivial oriented group involution on a group ring $RG$
relative to which the symmetric elements commute.  Thus $N=\ker\sigma$ has index $2$
in $G$ and, since $\sigma$ is not identically $1$, we must have $\ch R\ne 2$.
On the other hand, $\sigma_{|N}\equiv1$ and the restriction of $*$ to $N$ is an involution
$N\to N$ (by compatibility), so $N$ is abelian or an SLC group \cite[Theorem 2.4]{Jespers:06}.

\medskip

{\bfseries Assume that $N$ is abelian.}
\smallskip

Let $a$ be any element of $G\setminus N$.  If the elements of $G\setminus N$
commute, then $a$ and $na$ commute for any $n\in N$, so $a$ and $n$ commute
and it is clear that $G$ is abelian, a contradiction.  Thus we can always assume the
existence of elements $x,y\in G\setminus N$ with $xy\ne yx$.

Suppose that $x^*=x$ for all $x\notin N$.  Let $n\in N$ and $a\notin N$. Then $G=N\cup Na$
and $na\notin N$, so $na=(na)^*=a^*n^*=an^*$ giving $n^*=a^{-1}na$.
Let $\alpha_x,\alpha_y\in R_2$ and let $x$ and $y$ be elements of $G\setminus N$ that do not commute.
Since $x^*=x$ and $y^*=y$, the elements $\alpha_xx$, $\alpha_yy$ are in $\C{T}_3$, so they commute,
implying $0=\alpha_x\alpha_y(xy-yx)$.  So $\alpha_x\alpha_y=0$ and we have the situation
described in (1).
Henceforth, then, we assume the existence of an element $x\notin N$ with $x^*\ne x$.

Let $n\in N$ satisfy $n^*=n$.  Let $x\notin N$ be such that $x^*\ne x$.
Then $x-x^*\in\C{T}_4$ and this element commutes with $n$ (because $n\in\C{T}_1$).
The equation $n(x-x^*)=(x-x^*)n$ says $nx-nx^*=xn-x^*n$, so $nx+x^*n=xn+nx^*$.  Each
side of this equation is the sum of group elements.  In characteristic different from $2$,
the only possibility is $xn=nx$.

Let $n\in N$ with $n^*\ne n$ and again take $x\notin N$ with $x^*\ne x$.
The elements $n+n^*\in\C{T}_2$ and $x-x^*\in\C{T}_4$ commute, so
 $nx-nx^*+n^*x-n^*x^*=xn+xn^*-x^*n-x^*n^*$, that is,
\begin{equation*}
nx+n^*x+x^*n+x^*n^*=xn+xn^*+nx^*+n^*x^*.
\end{equation*}
Since $x\ne x^*$ and $n\ne n^*$ (and $\ch R\ne2$), $nx$ appears on the left side of this equation
(with nonzero coefficient), so it must be on the right and there are apparently
three possibilities---$nx=xn$, $nx=xn^*$ or $nx=n^*x^*$---but, in fact, just two because we shall
argue that the third alternative
implies one of the first two.  For this, assume $nx=n^*x^*=(xn)^*$.  We will show that
this implies $nx=xn$ or $nx=xn^*$.
The element $(nx)-(nx)^*=nx-xn$ is in $\C{T}_4$, so it commutes with $n+n^*$, that is,
\begin{equation*}
nxn+nxn^*-xn^2-xnn^*=n^2x-nxn+n^*nx-n^*xn,
\end{equation*}
which, rewritten, gives
\begin{equation*}
2nxn+nxn^*+n^*xn=n^2x+n^*nx+xn^2+xnn^*.
\end{equation*}
Since $nxn$ appears on the left side of this equation with coefficient $2$,
two terms on the right side are equal,
and equal to $nxn$.  Since $n^*\ne n$, $xn^2\ne xnn^*$ and $n^2x\ne nn^*x=n^*nx$, the
choices are
\begin{enumerate}
\item[i.] $nxn=n^2x=xn^2$;
\item[ii.] $nxn=n^2x=xnn^*$;
\item[iii.] $nxn=n^*nx=xn^2$;
\item[iv.] $nxn=n^*nx=xnn^*$.
\end{enumerate}
In cases i, ii and iii, $nx=xn$, and in case iv, since $n$ and $n^*$ commute in an SLC group,
we have $nxn=xnn^*=xn^*n$ and so $nx=xn^*$.  Thus $nx=n^*x^*$ implies $nx=xn$ or $nx=xn^*$ as claimed.

\smallskip
At this point, we have shown that for any $n\in N$ (whether or not $n^*=n$)
and for any $x\notin N$ with $x^*\ne x$, either $nx=xn$ or $n^*=x^{-1}nx$.

Now fix $a\notin N$ with $a^*\ne a$.  Let
\begin{equation*}
H =\{n\in N\mid an=na\} \text{ and } K=\{n\in N\mid n^*=a^{-1}na\}.
\end{equation*}
Each of these sets is a subgroup of $N$ (because $N$ is abelian) and $N=H\cup K$, so either
$H\subseteq K=N$ or $K\subseteq H=N$.
In the second case, $G=N\cup Na$ is abelian, contrary to assumption.

Thus we have the first case, $K=N$
and $n^*=a^{-1}na$ for all $n\in N$.  Let $x\notin N$.  Then
$x=na$ for some $n\in N$, so $x^*=a^*n^*=a^*a^{-1}na=sx$ with $s=a^*a^{-1}\ne1$ independent
of $x$.  In particular, $x=(x^*)^*=sx^*=s^2x$, so $s^2=1$.
Now $s\in N$, so $s^*=a^{-1}sa=a^{-1}a^*$ and $s=(s^*)^*=a(a^{-1})^*=asa^{-1}$ (because
$a^{-1}\notin N$).  It follows that $a^{-1}s=sa^{-1}$, so $s$ and $a$ commute
and $s^*=a^{-1}sa=s$.
As well, for any $n\in N$, $s$ and $n$ commute (because $s$ is in the abelian group $N$),
so $s(na)=(na)s$.  It follows that $s$ commutes with any $x\notin N$, so $s$ is central in $G$.

Now $G\setminus N=Na$ is not commutative (else $G$ is abelian),
so there exist elements
$x-x^*,y-y^*\in\C{T}_4$ with $xy\ne yx$.  Thus $x-x^*=x-sx=(1-s)x$ and $y-y^*=(1-s)y$ commute, so
$(1-s)^2xy=(1-s)^2yx$, that is, $2(1-s)xy=2(1-s)yx$,
equivalently,
\begin{equation*}
2xy+2syx=2yx+2sxy.
\end{equation*}
Since $xy\ne yx$, the only possibility is $\ch R=4$, $xy=syx$ and $yx=sxy$.  So the
commutator $(x,y)=x^{-1}y^{-1}xy=(yx)^{-1}(xy)=s$.  Let $n_1\in N$ and $x\notin N$.  Then
$x=n_2a$ for some $n_2\in N$ and the commutator $(n_1,x)=n_1^{-1}a^{-1}n_2^{-1}n_1n_2a
=n_1^{-1}a^{-1}n_1a$ because $N$ is abelian.  Neither $a^{-1}$ nor $n_1a$ are in $N$,
so the commutator of these two elements is $1$ or $s$; that is, $a^{-1}n_1a=n_1aa^{-1}=n_1$
or $a^{-1}n_1a=sn_1aa^{-1}=sn_1$, giving $(n_1,x)=n_1^{-1}n_1=1$ or $(n_1,x)=n_1^{-1}sn_1=s$.
It follows that $G'=\{1,s\}$ and we have the situation described by case~(2) of the theorem.

\medskip
{\bfseries Assume that $N$ is an SLC group} (with unique commutator $s$).
\smallskip

We first prove that the restriction of $*$ to $N$ is the canonical involution on $N$---see \eqref{eq0}.
Let $n\in N$ with $n^*=n$ and suppose that $n$ is not central, selecting $m\in M$
with $nm\ne mn$ (so $nm=smn$).  Then $n$ and $m+m^*$ commute, giving $nm+nm^*=mn+m^*n$ and hence $nm=m^*n$.
So $m^*=nmn^{-1}=sm$.  Now $n$ and $mn$ do not commute and so, similarly, $(mn)^*=smn$.
But $(mn)^*=n^*m^*=snm$, so $nm=mn$, a contradiction.  This shows that if $n=n^*$, then $n$ must be
central.

Suppose $n\in N$ is not central and choose $m\in N$ with $nm\ne mn$.  Then $m$ is also
noncentral, so both $n^*\ne n$ and $m^*\ne m$.  The elements $n+n^*$ and $m+m^*$ are
both in $\C{T}_2$, so they commute, giving
\begin{equation}\label{eq7}
nm+nm^*+n^*m+n^*m^*=mn+mn^*+m^*n+m^*n^*.
\end{equation}
Since $nm\ne nm^*$ and $n^*m\ne n^*m^*$, the element $nm$ appears on the left side,
so it appears on the right.  Now $nm\ne mn$ and $nm\ne m^*n^*$, else $nm=(nm)^*$ would
imply that $nm$ is central and $n$ and $m$ commute.  The possibilities are, therefore,
$nm=mn^*$ (implying $n^*=m^{-1}nm=sn$) or $nm=m^*n$ (implying $m^*=nmn^{-1}=sm$).  Suppose
$n^*=sn$ and apply what we have just learned to the noncentral elements $m$ and $nm$.
Either $m^*=sm$ or $(nm)^*=snm$, the latter giving $m^*n^*=snm$, so $sm^*n=snm=mn$ and
$m^*=sm$.  Similarly, if $m^*=sm$, we get $n^*=sn$ too.  Thus $m^*=sm$ for any noncentral
element $m$.

Let $n$ be central and $m$ noncentral.  Then $nm$ is noncentral, so $(nm)^*=snm$.
This gives $m^*n^*=snm$, $smn^*=snm=smn$, so $n^*=n$.  We have shown that $n$ is central
if and only if $n^*=n$ and hence $m$ is noncentral if and only if $m^*=sm$.  Indeed,
the restriction of $*$ to $N$ is canonical.
\smallskip

Suppose next that $x^*=x$ for all $x\notin N$.  Then for any $n\in N$ and any $a\notin N$,
$na\notin N$, so $na=(na)^*=a^*n^*=an^*$ giving $n^*=a^{-1}na$.  In particular,
$a^{-1}sa=s^*=s$, so $s$ is central in $G$.  Take $n\in N$ not central.  Then
$n^*=sn$, so $na=san$ (and $an=sna$).  Choose $m\in N$ with $nm\ne mn$.  Since $(nm)a\notin N$,
$[n(ma)]^*=[(nm)a]^*=nma$, while, since $ma\notin N$, $(ma)^*n^*=(ma)sn
=sman=sm(sna)=mna$.  These calculations show $[(n(ma)]^*\ne (ma)^*n^*$, a contradiction.
Thus there exists $a\notin N$ with $a^*\ne a$.

Now $s^*=s$, so  $s\in\C{T}_1$.  This element
commutes with $a-a^*\in\C{T}_4$, so $sa-sa^*=as-a^*s$, $sa+a^*s=as+sa^*$ giving
$sa=as$ (remember that $\ch R\ne2$).
Moreover, for any $x\notin N$, $x=na$ for some $n\in N$,
and $sx=s(na)=nsa=(na)s=xs$, so $s$ is central in $G$.

Let $n$ be any element of $N$.  We proceed exactly as in the $N$ abelian case.

If $n^*=n$, then $n$ commutes with $a-a^*$, so $na=an$.
For example, $aa^*$ is invariant under $*$, so $(aa^*)a=a(aa^*)$ implying $aa^*=a^*a$.

If $n^*\ne n$, the elements $n+n^*$ and $a-a^*$ commute, so
\begin{equation*}
na-na^*+n^*a-n^*a^*=an+an^*-a^*n-a^*n^*,
\end{equation*}
which, rewritten, says,
\begin{equation*}
na+n^*a+a^*n+a^*n^*=an+an^*+na^*+n^*a^*,
\end{equation*}
giving rise apparently to three possibilities,
\begin{equation}\label{eq1}
na=an, \quad na=an^*, \quad na=n^*a^*.
\end{equation}
Just as before, we argue that the third alternative implies
one of the first two.  So assume
$na=n^*a^*$.  Then $n\ne n^*$ means $n^*=sn$ and
$na=n^*a^*=sna^*$ gives $a^*=sa$ too.  Let $b=an$.  If $b=b^*$,
then $an=n^*a^*=(sn)(sa)=na$.
If $b^*\ne b$, then just as with $a$, either $nb=bn$ or $nb=bn^*$ or $nb=n^*b^*$.
If $nb=bn$, then $nan=an^2$, so $na=an$.
If $nb=bn^*$, then $nan=ann^*=an^*n$, so $na=an^*$.
If $nb=n^*b^*$, then $nan=n^*n^*a^*=(sn)n^*(sa)=nn^*a$, so $an=n^*a$ and, applying
the involution $*$, $n^*a^*=a^*n$; that is $(sn)(sa)=(sa)n$, so $na=a(sn)=an^*$.
All this shows that for any $n\in N$ (whether or not $n^*=n$), either $na=an$ or $na=an^*$.

\smallskip
We distinguish two cases.

{\bfseries\boldmath Case I. Suppose $na=an$ for all $n\in N$.}  Then it is easy to see that $a$ is
central, $G=N\langle a\rangle$
is product of the groups $N$ and $\langle a\rangle$ and so $s$ is in fact a unique nonidentity
commutator in $G$.  We claim we may assume $a^*=sa$, without loss of generality.

There exist $n,m\in N$ with $nm\ne mn$.  This implies $n^*\ne n$
and $m^*\ne m$, so $n^*=sn$ and $m^*=sm$.  Let $x=na$, $y=ma$ and note
that $xy=(nm)a^2$ while $yx=(mn)a^2$, so $xy\ne yx$.

We now argue, by way of contradiction, that either $x^*=x$ or $y^*=y$.
Thus we suppose that $x^*\ne x$ and $y^*\ne y$ and conclude that
the $\sharp$-symmetric elements $x-x^*$ and $y-y^*$ must
commute, giving
\begin{equation*}
xy-xy^*-x^*y+x^*y^*=yx-yx^*-y^*x+y^*x^*;
\end{equation*}
that is,
\begin{equation}\label{eq9}
xy+x^*y^*+yx^*+y^*x=yx+y^*x^*+xy^*+x^*y.
\end{equation}
Suppose $x^*y=yx^*$.  Then $x^*yy^*=yx^*y^*$.  Since $yy^*$ is central in $N$,
it is central in $G$, so $yy^*x^*=yx^*y^*$, giving $y^*x^*=x^*y^*$ and so $xy=yx$,
which is not true.  So $x^*y\ne yx^*$, hence $yx^*=sx^*y$ and, similarly, $y^*x=sxy^*$.
Now \eqref{eq9} reads
\begin{equation}\label{eq10}
xy+x^*y^*+sx^*y+sxy^*=yx+y^*x^*+xy^*+x^*y.
\end{equation}
If $x^*=sx$, then $a^*n^*=sna$, so $sa^*n=san$ and $a^*=a$, a contradiction.
Thus $x^*\ne sx$ and, for the same reason, $y^*\ne sy$.  It follows that the three
elements $x^*y^*, sx^*y, sxy^*$ are distinct and so $xy$ appears on the left side
of \eqref{eq10} and hence on the right.

The only possibility is $xy=y^*x^*=(xy)^*$; that is $xy$ is a $*$-symmetric
element of $N$.  This means $xy$ is central in $N$, so $xy$ is central in $N\langle a\rangle=G$
implying $xy=yx$, a contradiction.  It follows, as claimed, that either $x^*=x$ or $y^*=y$.

The first possibility says
$a^*(sn)=na=an$, so $a^*=sa$, and we reach the same conclusion
if $y^*=y$.  Thus $a^*=sa$.

Finally, let $n,m\in N$ with $nm\ne mn$.  Then neither $n$ nor $m$ is central in the SLC
group $N$, so $n^*=sn$ and $m^*=sm$.  Let $x=na$ and note that $x^*=sn^*a=na=x$
and, similarly, $y^*=y$.  Note too that $xy\ne yx$.  Let $\alpha_x,\alpha_y\in R_2$.  Then $\alpha_x x$ and
$\alpha_yy$ are in $\C{T}_3$.  So they must commute, giving $\alpha_x\alpha_y(xy-yx)=0$.
Since $xy\ne yx$, it follows that $\alpha_x\alpha_y=0$.  Thus $R_2^2=\{0\}$ and
the situation is as described in (3).
\smallskip

{\bfseries\boldmath Case II. Suppose that for each $a\notin N$ with $a^*\ne a$,
we have $an\ne na$ for some $n\in N$}.  This implies that if $a\notin N$ and $a=a^*$,
then $na=an$ for all $n\in N$ and, since $G=N\cup Na$, such $a$ is central in $G$.
Remembering that $n^*=n$ and $a^*\ne a$ implies $na=an$ (consider the commutativity
of $n$ and $a-a^*$), it follows that any $g\in G$ with $g^*=g$ is central.  For instance,
$s$ is central in $G$.

Now fix $a\notin N$ with $a^*\ne a$.  There exists $n\in N$ with $an\ne na$.
Remember that this implies $na=an^*$.
Also, $n^{-1}$ and $a$ cannot commute, so
$(n^{-1})^*\ne n^{-1}$ and $n^{-1}a=a(n^{-1})^*
=a(sn^{-1})$, so $s=a^{-1}n^{-1}an=(a,n)$, the commutator.
We claim this implies that $s$ is the only nonidentity commutator in $G$.  This is already
the case in the SLC group $N$ so, noting that $G=N\cup Na$,
there are just two types of commutators to consider.
If $m,n\in N$, then $m(na)=(mn)a$ whereas
\begin{equation*}
(na)m=nam=\begin{cases}
             nma & \text{if $am=ma$} \\
             snma & \text{if $am\ne ma$}
             \end{cases}
\end{equation*}
and, since $nm=mn$ or $nm=smn$, it is clear that $(na)m=nma$ or $(na)m=snma$, so
the commutator $(m,na)=1$ or $s$.  Similarly, it is easy to show that any commutator
$(na,ma)$ with $n,m\in N$ is also $1$ or $s$.

Again suppose that $na\ne an$ for some $n\in N$ and $a\notin N$.
Then $(a,n)=s$.  Also $a^*$ and $n$ cannot commute for
$a^*n=na^*$ implies $a^*na=na^*a=a^*an$, which is false.
So $s=(a^*,n)$ too.
Now $na\ne (na)^*$; otherwise, $na=a^*n^*=a^*sn$, which says
$san=a^*sn$ and $a=a^*$, which is not true.  Thus $na-(na)^*$, which is
$na-a^*n^*=na-sa^*n=na-na^*=n(a-a^*)$ and $a-a^*$ are each in $\C{T}_4$, so they commute.  This gives
\begin{equation*}
ana-ana^*-a^*na+a^*na^*=n(a^2-2aa^*+(a^*)^2),
\end{equation*}
which we rewrite as
\begin{equation*}
sna^2-snaa^*-sna^*a+sn(a^*)^2=na^2-2naa^*+n(a^*)2
\end{equation*}
and then
\begin{equation}\label{eq11}
sna^2+sn(a^*)^2+2naa^*=na^2+n(a^*)^2+2snaa^*.
\end{equation}
If $a^*\ne sa$, then $sna^2\ne naa^*$ and $naa^*$ appears on the left side of this
equation while it cannot appear on the right.  We conclude that $a^*=sa$ in which case
\eqref{eq11} becomes $4sna^2=4na^2$, so $R$ must have characteristic~$4$.

We have shown that if $a\notin N$ with $a^*\ne a$, then such an element is not central
and $a^*=sa$.  It follow that the involution on $G$ is given by \eqref{eq0} and hence $G$
is SLC:  if $gh=hg$ and neither $g$ nor $h$ nor $gh$ is central, then $sgh=(gh)^*=h^*g^*=(sh)(sg)
=hg=gh$, a contradiction.  The situation is as described in (4).

\end{proof}


\providecommand{\bysame}{\leavevmode\hbox to3em{\hrulefill}\thinspace}
\providecommand{\MR}{\relax\ifhmode\unskip\space\fi MR }
\providecommand{\MRhref}[2]{%
  \href{http://www.ams.org/mathscinet-getitem?mr=#1}{#2}
}
\providecommand{\href}[2]{#2}

\end{document}